\documentclass[12pt]{amsart}
\usepackage{amssymb} 
\usepackage{amsthm}
\usepackage{amscd}
\usepackage{multicol}
\usepackage[dvips]{color}
\newtheorem{thm}{Theorem}[section]
\newtheorem{lem}[thm]{Lemma}
\newtheorem{lem-dfn}[thm]{Lemma-Definition}
\newtheorem{prop}[thm]{Proposition}
\newtheorem{cor}[thm]{Corollary}

\theoremstyle{definition}
\newtheorem{defn}[thm]{Definition}
\newtheorem{exam}[thm]{Example}
\newtheorem{ex}[thm]{Example}

\newtheorem{quest}[thm]{Question}

\newtheorem*{acknowledgement}{Acknowledgement}
\newtheorem{noname}[thm]{}
\theoremstyle{remark}

\newtheorem{rem}[thm]{Remark}
\numberwithin{equation}{section}

\newcommand{\thmref}[1]{Theorem~\ref{#1}}
\newcommand{\lemref}[1]{Lemma~\ref{#1}}
\newcommand{\corref}[1]{Corollary~\ref{#1}}
\newcommand{\proref}[1]{Proposition~\ref{#1}}
\newcommand{\remref}[1]{Remark~\ref{#1}}

\newcommand{\defref}[1]{Definition~\ref{#1}}

\DeclareMathOperator{\Spec}{Spec}
\DeclareMathOperator{\spec}{Spec}

\DeclareMathOperator{\Hom}{Hom}

\DeclareMathOperator{\Soc}{Soc}
\DeclareMathOperator{\chara}{char}

\DeclareMathOperator{\di}{div}

\DeclareMathOperator{\G}{G}
\DeclareMathOperator{\GG}{\Gamma G}

\DeclareMathOperator{\nr}{nr}
\DeclareMathOperator{\br}{\bar r}

\newcommand{\m}{\mathfrak m}
\newcommand{\MM}{\mathfrak M}

\newcommand{\PP}{\mathbb P}

\newcommand{\Z}{\mathbb Z}
\newcommand{\Q}{\mathbb Q}

\newcommand{\bbZ}{\ensuremath{\mathbb Z}}

\newcommand{\cE}{\mathcal E}

\newcommand{\cO}{\mathcal O}

\renewcommand{\:}{\colon}

\newcommand{\ol}[1]{\overline {#1}}

\newcommand{\defset}[2]{{\left\{#1\,\left| \,#2 \right. \right\}}}

\begin{document}
\title{Gorenstein Normal tangent cones of integrally closed ideals in two-dimensional normal singularities
}

\author{Tomohiro Okuma}
\address[Tomohiro Okuma]{Department of Mathematical Sciences, 
Yamagata University,  Yamagata, 990-8560, Japan.}
\email{okuma@sci.kj.yamagata-u.ac.jp}
\author{Kei-ichi Watanabe}
\address[Kei-ichi Watanabe]{Department of Mathematics, College of Humanities and Sciences, 
Nihon University, Setagaya-ku, Tokyo, 156-8550, Japan and 
Organization for the Strategic Coordination of Research and Intellectual Properties, Meiji University
}
\email{watnbkei@gmail.com}
\author{Ken-ichi Yoshida}
\address[Ken-ichi Yoshida]{Department of Mathematics, 
College of Humanities and Sciences, 
Nihon University, Setagaya-ku, Tokyo, 156-8550, Japan}
\email{yoshida.kennichi@nihon-u.ac.jp}

\thanks{The first-named author was partially supported by JSPS Grant-in-Aid for Scientific Research (C) Grant Number 21K03215. 
The second-named author was partially supported by JSPS Grant-in-Aid for Scientific Research (C) Grant Number 23K03040. 
The third named author was partially supported by JSPS Grant-in-Aid for Scientific Research (C) Grant Number 24K06678.}

\pagestyle{plain}

\begin{abstract} 
Let $(A,\m)$ be a two-dimensional excellent normal Gorenstein 
local domain containing an algebraically closed filed. 
Let $I =H^0(X,\mathcal{O}_X(-Z)) \subset A$ be 
an $\m$-primary integrally closed ideal represented by an anti-nef cycle $Z$ 
on some resolution $X\to \Spec A$. 
In this paper, we prove that $\overline{G}(I)$ is Gorenstein if and only if 
it is Cohen-Macaulay and $(r-1)Z^2+K_XZ=0$, where 
$r=\br(I)$ denotes the normal reduction number of $I$ and 
$K_X$ denotes the canonical divisor on $X$. 
\end{abstract}

\maketitle

\section*{Introduction}

Throughout this paper, let $(A,\m)$ be a Noetherian commutative 
local ring with the unique maximal ideal $\m$ and the residue field 
$k=A/\m$. 
Furthermore, we assume that $A$ is a two-dimensional 
excellent normal local domain (which is not regular) 
containing an algebraically closed field $k$ unless otherwise specified. 
 Then for any $\m$-primary integrally closed ideal $I \subset A$, there exists a resolution $X\to \spec A$ of singularities of $\Spec A$, and an anti-nef cycle $Z$ on $X$ such that $I$ can be
represented  by $Z$ as follows:  
$I\mathcal{O}_X=\mathcal{O}_X(-Z)$ and $I=H^0(X,\mathcal{O}_X(-Z))$. 
Then we denote it by $I=I_Z$.

\par \vspace{3mm}
In \cite{OWY4}, the authors introduced the following two 
normal reduction numbers. 
For any $\m$-primary integrally closed ideal $I \subset A$ and 
its minimal reduction $Q$, we put  
\begin{eqnarray*}
\nr(I) &=& \min\{r \in \bbZ_{\ge 1} \,|\,  \overline{I^{r+1}}=Q \overline{I^r}\}, \\
\br(I)&=&\min\{r \in \bbZ_{\ge 1} \,|\, \overline{I^{n+1}}=Q \overline{I^n} 
\; \text{for all $n \ge r$}\}. 
\end{eqnarray*}
Then $I$ is a $p_g$-ideal if and only if $\br(I)=1$.

\par \vspace{3mm}
For an ideal $I \subset A$, we put 
\[
G(I)=\bigoplus_{n=0}^{\infty} I^n/I^{n+1}, \qquad 
\overline{G}(I)=\bigoplus_{n=0}^{\infty} \overline{I^n}/\overline{I^{n+1}}. 
\] 
Then $G(I)$ (resp. $\overline{G}(I)$) 
is called the associated graded ring (resp. 
the normal tangent cone) of $I$. 
Many authors have studied the Cohen-Macaulayness and 
the Gorensteinness of $G(I)$. 
However, several ring-theoretic properties of $\overline{G}(I)$ 
seems to be not known in spite of its importance from geometric point of view. 
So in \cite{OWY6} we give a criterion for $\overline{G}(I)$ to be Gorenstein 
in the case of elliptic ideals $I$. 
So it is natural to ask the following question. 

\par \vspace{2mm} \par \noindent 
{\bf Question.}
Suppose that  $\overline{G}(I)$ is Cohen-Macaulay. 
When is it Gorenstein? 

\par  \vspace{2mm} 
 The following theorem gives an answer to the question above. 
In this theorem,  our main contribution is an equivalence between (1) 
and (4). 
This gives a geometric characterization of Gorensteinness of $\overline{G}(I_Z)$. 
Indeed, an equivalence of $(1)$, $(2)$ and (3) follows from 
\cite[Theorem 4.4]{HKU}. 
Moreover, Heinzer et. al \cite[Theorem 3.12]{HKU} 
proved that (2) implies the Cohen-Macaulayness of 
$\overline{G}(I)$. 

\par \vspace{3mm} \par \noindent 
{\bf Theorem (see Theorem $\ref{Main}$ and Corollary \ref{VerNTC}).}
Let $(A,\m)$ be a two-dimensional excellent normal Gorenstein local 
domain containing a field $k=\overline{k}$. 
Let $I=I_Z$ with $r=\br(I) \ge 2$ and $Q$ its minimal reduction of $I$. 
 We put $L_n=Q+\overline{I^n}$ for every integer $n \ge 1$. 
Assume that $\overline{G}(I)$ is Cohen-Macaulay. 
Then the following conditions are equivalent$:$
\begin{enumerate}
\item $\overline{G}(I)$ is Gorenstein. 
\item 
$Q \colon L_n=L_{r+1-n}$ for every $n=1,2,\ldots, r$. 
\item
$\ell_A(A/L_n)+\ell_A(A/L_{r+1-n})=e_0(I)$ for  every $n=1,2,\ldots,\lceil \frac{r}{2} \rceil$. 
\item $(r-1)Z^2+K_XZ=0$. 
\end{enumerate}
When this is the case, 
\begin{enumerate}
\item[$(i)$] $\ell_A(A/I) = \ell_A(\overline{I^{r}}/Q\overline{I^{r-1}}) 
\le p_g(A)+2-r$.   
\item[$(ii)$] If we put $I_{r-1}=\overline{I^{r-1}}$, then it is an elliptic ideal and $\overline{G}(I_{r-1})$ is Gorenstein.  
\end{enumerate}

\par \vspace{3mm}
In Section 3, we investigate the Gorensteinness for $\overline{G}(\m)$  
of the maximal ideal $\m$ in a Brieskorn hypersurface $A$ of type $(a,b,c)$:
\[
A=K[x,y,z]_{(x,y,z)}/(x^a+y^b+z^c), 
\]
where $2 \le a \le b \le c$ and $\chara K=0$ or $\chara K=p$ does not divide $abc$.
Indeed, we prove 
\par \vspace{3mm} \par \noindent 
{\bf Theorem (see Theorem $\ref{BmaxThm}$)} 
Let $A$ be a Brieskorn hypersurface of type $(a,b,c)$ and put $r=n_{a-1}:=\lfloor \frac{(a-1)b}{a} \rfloor$, $d=\gcd(a,b)$. 
Then the following conditions are equivalent$:$
\begin{enumerate}
\item $\overline{G}(\m)$ is Gorenstein. 
\item $\ell_A(A/L_n)+\ell_A(A/L_{r+1-n})=a$ for every $n=1,2,\ldots, r$.  
\item $(r-1)Z^2+K_XZ=0$.  
\item $b \equiv 0$ or $d \pmod{a}$.
\end{enumerate}

\par \vspace{3mm}
In general, there exist infinitely many $p_g$-ideals $I$ for which $\overline{G}(I)$ is Gorenstein 
(i.e., good $p_g$-ideals)  in a given local ring $A$  (see \remref{r:pginfty}). 
In Section 4,  we introduce the weighted dual graph $\Gamma(I)$ of $I$ and prove that the combinatorial type, i.e., the weighted dual graph, of $I$ 
such that $\br(I)\ge 2$ and $\overline{G}(I)$ is Gorenstein is finite (\thmref{t:GG}). 
Moreover, for homogeneous hypersurface $A$ of degree $d \le 5$, we can classify all ideals 
$I=I_Z \subset A$ such that $\overline{G}(I)$ is Gorenstein (Theorem \ref{Main: d-4,5}).

\section{Preliminaries}

\par 
Let $A$ be a commutative Noetherian ring, and 
let $I \subset A$ be an ideal. 
An element $a \in A$ is said to be integral over $I$ if there exists a 
monic polynomial 
\[
f(X)= X^n +c_1X^{n-1}+\cdots + c_n \;\; (c_i \in I^i,\, i=1,2,\ldots,n)
\]
such that $f(a)=0$. 
The ideal of all elements that are integral over $I$ is called the 
integral closure of $I$, denoted by $\overline{I}$. 
If $\overline{I}=I$, then $I$ is said to be {\it integrally closed}. 
An ideal $I$ is called {\it normal} if all powers of $I$ are integrally closed.  
\par 
In what follows, let $(A,\m)$ be an excellent  two-dimensional 
normal local domain containing an algebraically closed field 
$k= A/\m$, where $\m$ denotes the unique maximal ideal of $A$,
 and $I \subset A$ an $\m$-primary integrally closed ideal. 
A divisor on a resolution space whose support is contained 
in the exceptional set is called a cycle. 
Then in our situation, there exist a resolution of singularities 
$f \colon X \to \Spec A$ and an anti-nef cycle $Z$ on $X$ such that 
$I\mathcal{O}_X=\mathcal{O}_X(-Z)$ and $I=H^0(X,\mathcal{O}_X(-Z))$. 
Then we say that $I$ is \textit{represented} by $Z$ on $X$, 
and denote it by $I=I_Z$. Note that $\overline{I^k}=I_{kZ}$. 

\par \vspace{2mm}
\subsection{Riemann-Roch formula and normal reduction numbers}

\begin{defn}[\textrm{cf. \cite[Definition 2.1]{OWY4}}] \label{qkI}
For $I=I_Z$, we define 
\[
q(I) = \dim_k H^1(X,\mathcal{O}_X(-Z)). 
\]
Then $q(kI):=q(\overline{I^k})$ is a non-negative integer for every  $k \ge 0$, and 
there exists a decreasing sequence 
\[
p_g(A):=q(0 I) \ge q(I) \ge q(2I) \ge q(3I) \ge \cdots \ge 0.  
\] 
We put $q(\infty I):=q(kI)$ for sufficiently large $k$. 
\par 
Moreover, $p_g(A)=\dim_k H^1(X,\mathcal{O}_X)$ 
is called the {\it geometric genus} of $A$. 
Note that $q(kI)$'s are independent of the choice of the resolution.  
\end{defn}

\par
The following formula was proved in \cite{kato} in the complex case, but it holds in any characteristic \cite{WY}. 

\begin{thm}[\textbf{Kato's Riemann Roch formula}] \label{RRformula}
Assume $I=I_Z$. Let $K_X$ denote the canonical divisor on $X$. 
Then we have 
\[
\ell_A(A/I)+q(I)=\chi(Z)+p_g(A), \;\text{where}\;\; \chi(Z)=-\dfrac{Z^2+K_XZ}{2}
\]
\end{thm}

\par \vspace{2mm}
For any ideal $I$, let $Q$ be its minimal reduction. 
Then we put 
\[
r_Q(I)=\min\{r \in \Z_{\ge 1}\,|\, I^{r+1}=QI^r\}. 
\]
In general, $r_Q(I)$ depends on the choice of $Q$. So
we define 
\[
r(I)=\min\{r_Q(I) \,|\, \text{$Q$ is a minimal reduction of $I$}\}. 
\]
Then $r(I)$ is called a \textit{reduction number} of $I$.  
We now recall the definition of two normal reduction numbers. 
By virtue of \cite[Lemma 4.1]{Hun}, $\nr(I)$ and $\br(I)$ are independent on the choice of $Q$. 

\begin{defn}[\textrm{cf. \cite[Definition 2.5]{OWY4}}] 
\label{NorRedNum}
For $I=I_Z$ and its minimal reduction $Q$, 
we put
\begin{eqnarray*}
\nr(I)&=& \min\{r \in \mathbb{Z}_{\ge 1}  \,| \,  \overline{I^{r+1}}=Q\overline{I^r} \}, \\[1mm]
\br(I) &=& \min\{r \in \mathbb{Z}_{\ge 1}  \,| \,  \overline{I^{n+1}}=Q\overline{I^n} \; \text{for all $n \ge r$}\}. 
\end{eqnarray*}
\par 
A positive integer $\br(I)$ (resp. $\nr(I)$) is called 
the {\it normal reduction number} 
(resp. the {\it relative normal reduction number}) 
\end{defn}

\par \vspace{2mm}
The following relationship is very important in our theory. 
\begin{prop}[\textrm{cf. \cite[Proposition 2.2]{OWY5}}] \label{q(nI)formula}
For $I=I_Z$, we put $q_{n}=q(nI)$. Then 
\begin{enumerate}
\item $\ell_A(\overline{I^{n+1}}/Q\overline{I^n})
=(q_{n-1}-q_n)-(q_n-q_{n+1})$. 
\item $\nr(I)= \min\{n \in \bbZ_{\ge 1} \,| \, 
q_{n-1}-q_n = q_n-q_{n+1} \}$. 
\item $\br(I)=\min\{n \in \bbZ_{\ge 1} \,|\, q_{n-1}=q_n  \}$. 
\end{enumerate}
\end{prop}
\par \vspace{2mm}
Put $I_k = \overline{I^k}$. Since $q (n I_k) = q((kn) I)$, we can see the following, 
which will be used in \ref{VerNTC}.

\begin{cor}\label{br(Ik)} 
$\br(I_k) = \left\lceil \dfrac{r-1}{k} \right\rceil +1$.  
\end{cor}

\par \vspace{3mm}
\subsection{Normal tangent cones}
Let us recall the definition of normal tangent cones, which is the main 
target in this paper. 

\begin{defn}[\textbf{Normal Tangent Cone}] \label{def:cones}
For any $\m$-primary integrally closed ideal $I \subset A$,  
\[
\overline{G}(I) = \bigoplus_{n \ge 0} \overline{I^n}/\overline{I^{n+1}}
\]
is called the {\it normal tangent cone} of $I$. 
\end{defn} 

\par \vspace{3mm}
For an ideal $I$ of $A$,    
$G(I) = \oplus_{n \ge 0} I^n/I^{n+1}$ 
is called the {\it associated graded ring} or the {\it tangent cone} of $I$. 
We notice that $\overline{G}(I)$ is \textit{not} an integral closure of $G(I)$ even if $A$ is a normal domain and $I$ is integrally closed.  
However, if we consider the extended Rees algebra 
\[
\mathcal{R}':=\mathcal{R}'(I)=  A[It, t^{-1}] = \bigoplus_{n \in \mathbb{Z}} I^n t^n \subset A[t,t^{-1}] 
\]
and its integral closure $\overline{\mathcal{R}'}$ of $\mathcal{R}'$ 
in its fraction field, then 
\[
G(I) \cong \mathcal{R}'/t^{-1}\mathcal{R}', \qquad 
\overline{G}(I) \cong \overline{\mathcal{R}'}/t^{-1} \overline{\mathcal{R}'}. 
\]

\par \vspace{3mm}
In Section 2, we give an answer to this question in our context. 

\begin{quest}
Put $I=I_Z \subset A$. When is $\overline{G}(I)$ Gorenstein?
\end{quest}

\par  \vspace{3mm}
Valabrega-Valla type theorem yields a criterion for $\overline{G}(I)$ 
to be Cohen-Macaulay. 

\begin{thm}[\textrm{cf. \cite{HuMa, VV}}] \label{Ideal-GCM}
Let $A$, $I=I_Z$ be as above. 
Then 
\begin{enumerate}
\item $\overline{G}(I)$ is Cohen-Macaulay if and only if $Q \cap \overline{I^n} =Q\overline{I^{n-1}}$ holds for some minimal reduction $Q$ of $I$ and for all $n \ge 2$. 
\item For any minimal reduction $Q$ of $I$, we have $Q \cap \overline{I^2}=QI$ (Huneke-Itoh). 
\item If $\overline{G}(I)$ is Cohen-Macaulay, then $\nr(I)=\br(I)$. 
\end{enumerate}
\end{thm}

\par \vspace{2mm}
Note that $\nr(I)=1$ does not necessarily imply 
the Cohen-Macaulayness of $\overline{G}(I)$. 
%
\par \vspace{2mm}
\subsection{Geometric ideals and singularities}
We recall the notion of $p_g$-ideals and elliptic ideals,  
which was introduced by the 
authors in \cite{OWY1} or \cite[Theorem 3.2]{ORWY}.  

\begin{defn}[\textbf{$p_g$-ideal}] \label{def:pgideal}
Assume that $I = I_Z$. 
Then $I$ is called the \textit{$p_g$-ideal} if one of the following equivalent conditions is satisfied: 
\begin{enumerate}
\item $\br(I)=1$. 
\item $q(I)=p_g(A)$. 
\item $I$ is normal and $I^2=QI$ for some 
minimal reduction $Q$ of $I$. 
\item $\overline{G}:=\overline{G}(I)$ is Cohen-Macaulay and $a(\overline{G})<0$, where
 $a(\overline{G})=\max\{n \in \bbZ \,|\, [H^{2}_{\MM}(\overline{G})]_n \ne 0\}$, the \textit{$a$-invariant} 
of $\overline{G}$, and  $\MM$ is the maximal homogeneous ideal.
\end{enumerate}
\end{defn}

\begin{defn}[\textbf{Elliptic ideal}] \label{def:elliptic} 
An ideal $I=I_Z$ is called the \textit{elliptic ideal}
if one of the following equivalent conditions is satisfied: 
\begin{enumerate}
\item $\br(I)=2$. 
\item $q(I) =q(\infty I) < p_g(A)$.
\item $\overline{G}:=\overline{G}(I)$ is Cohen-Macaulay and $a(\overline{G})=0$. 
\end{enumerate}
\end{defn}

\par \vspace{2mm}
A normal local domain $A$ is called a \textit{rational singularity} if 
$p_g(A)=0$. 
Lipman's theorem says that any $\m$-primary integrally closed ideal 
$I$ in a rational singularity is a $p_g$-ideal. 

\par  \vspace{2mm}
A normal local domain $A$ is called an \textit{elliptic singularity} 
if $\chi(D)\ge 0$ for all cycles $D > 0$ and $\chi(F)=0$ for some cycle $F > 0$; see Wagreich \cite[pp.428]{Wa}. 
Okuma \cite{Ok} showed that $\br(I) \le 2$ for any $I=I_Z \subset A$. 
In fact, any $\m$-primary integrally closed ideal $I$ in an 
elliptic singularity is either a $p_g$-ideal or an elliptic ideal. 

\par  \vspace{2mm}
A normal local domain $A$ is called a \textit{strongly elliptic singularity} 
if $p_g(A)=1$. 
Any non-$p_g$-ideal $I=I_Z$ in a strongly elliptic singularity is a \textit{strongly elliptic ideal}, 
that is, it is an elliptic ideal with $\ell_A(\overline{I^2}/QI)=1$.

\subsection{Gorensteinness for normal tangent cones of geometric ideals}
\par  \vspace{2mm}
Gorensteinness of normal tangent cones of some geometric ideals ($p_g$-ideals and elliptic ideals) 
are known. 

\begin{thm}[  \textrm{cf. \cite{OWY6}}]  
\label{pgCM}
Assume that $A$ is Gorenstein and that $I$ is a $p_g$-ideal. 
Then the following conditions are equivalent$:$
\begin{enumerate}
\item $\overline{G}(I)$ is Gorenstein.
\item  $I$ is \emph{good} 
$($i.e., $I^2=QI$ and $Q:I=I$ for some (every) minimal reduction $Q$ of $I$$)$. 
\item $K_XZ=0$. 
\end{enumerate}
\end{thm}

\begin{thm}[\textrm{\cite{OWY6}}] \label{EllipGor}
Assume that $A$ is Gorenstein and $I$ is an elliptic ideal. 
For any minimal reduction $Q$ of $I$,  
the following conditions are equivalent$:$
\begin{enumerate}
\item $\overline{G}(I)$ is Gorenstein. 
\item $Q \colon I=Q+\overline{I^2}$. 
\item $\ell_A(\overline{I^2}/QI)=\ell_A(A/I)$. 
\item $Z^2+K_XZ=0$, that is, $\chi(Z)=0$
\end{enumerate}
When this is the case, $\ell_A(A/I) \le p_g(A)$.
\end{thm}

\par \vspace{2mm}
In the next section, we prove the main theorem, which generalize these characterization 
in the case of $\br(I) \ge 3$.

\section{Gorensteinness of normal tangent cones}

In this section, we keep the notation as in the previous section. 
The main aim of this section is to prove the following theorem, 
which gives a criterion for Gorensteinness of normal tangent cones 
 of $I=I_Z$ in terms of the cycle $Z$.

\par \vspace{2mm}
For an ideal $I=I_Z$ and a minimal reduction $Q$ of $I$, 
we put 
\[
L_n:=Q+\overline{I^n} 
\]
for every integer $n \ge 1$. 

\begin{thm} \label{Main}
Let $(A,\m)$ be a two-dimensional excellent normal Gorenstein local 
domain containing a field $k=\overline{k}$. 
Let $I=I_Z$ with $r=\br(I) \ge 2$ and $Q$ its minimal reduction of $I$. 
Assume that $\overline{G}(I)$ is Cohen-Macaulay. 
Then the following conditions are equivalent$:$
\begin{enumerate}
\item $\overline{G}(I)$ is Gorenstein. 
\item 
$Q \colon L_n=L_{r+1-n}$ for every $n=1,2,\ldots, r$. 
\item
$\ell_A(A/L_n)+\ell_A(A/L_{r+1-n})=e_0(I)$ for  every $n=1,2,\ldots,\lceil \frac{r}{2} \rceil$. 
\item $(r-1)Z^2+K_XZ=0$  or, equivalently, 
\[
(r-2) Z^2 = 2 \cdot \chi(Z).
\] 
\end{enumerate}
When this is the case, 
$\ell_A(A/I) = \ell_A(\overline{I^{r}}/Q\overline{I^{r-1}}) 
\le p_g(A)+2-r$.   
\end{thm}

\par \vspace{3mm}
In what follows, we prove the theorem above. 
We put $Q=(a,b)$ and 
assume that $\overline{G}(I)$ is Cohen-Macaulay. 
If we put $a^{*}=a+\overline{I^2}$, 
$b^{*}=b+\overline{I^2} \in \overline{G}(I)_1$, 
then $a^{*}$, $b^{*}$ forms a regular sequence. 
We put 

\begin{eqnarray*}
B &:= & \overline{G}(I)/(a^{*},b^{*})\overline{G}(I) \\
& \cong & \bigoplus_{n=0}^{\infty} \dfrac{\overline{I^n}}{Q\overline{I^{n-1}}+\overline{I^{n+1}}}   
=  \bigoplus_{n=0}^{\infty} \dfrac{\overline{I^n}}
{Q \cap \overline{I^{n}}+\overline{I^{n+1}}} 
 \cong  \bigoplus_{n=0}^{\infty}  
\dfrac{(Q+ \overline{I^n})+\overline{I^{n+1}}}{Q +\overline{I^{n+1}}} \\
& = & A/I \oplus I/L_2 \oplus L_2/L_3 \oplus \cdots \oplus L_{r-1}/L_r \oplus L_r/Q. 
\end{eqnarray*}
Then $B$ is Gorenstein if and only if so is $\overline{G}(I)$. 
We put $b_n=\ell_A(B_n)$ for every integer   $n \ge 0$. 

\begin{lem} \label{B-lem}
We have
\begin{enumerate}
\item $b_0=\ell_A(A/I)$ and $b_r=\ell_A(\overline{I^r}/Q\overline{I^{r-1}})$. 
\item $b_0+b_1+\cdots+b_r = \ell_A(B)=e_0(I)=\ell_A(A/Q)$. 
\end{enumerate}
\end{lem}

\par \vspace{2mm}
 We can describe $\ell_A(A/L_n)$ and $\ell_A(A/(Q:L_n))$ 
in terms of $b_i$'s.  

\begin{lem} \label{sum-B}
For an integer $n$ 
with $1 \le n \le r$, we have 
\begin{enumerate}
\item $\ell_A(A/L_n)=b_0+b_1+\cdots+b_{n-1}$. 
\item $\ell_A(A/(Q:L_n))=b_n+b_{n+1}+\cdots + b_r$. 
\item $\ell_A(\overline{I^n}/Q\overline{I^{n-1}})=b_n+b_{n+1}+\cdots+b_r$. 
\end{enumerate}
\end{lem}

\begin{proof}
(1) $\ell_A(A/L_n)=\ell_A(A/L_1)+\ell_A(L_1/L_2)+\cdots +\ell_A(L_{n-1}/L_n)=b_0+b_1+\cdots+b_{n-1}$. 
\vspace{3mm}
(3) As $\overline{G}(I)$ is Cohen-Macaulay, we have 
\begin{eqnarray*}
\ell_A(\overline{I^n}/Q\overline{I^{n-1}}) 
& = & \ell_A(\overline{I^n}/Q \cap \overline{I^{n}}) 
 = \ell_A(L_n/Q) 
 =  \ell_A(B)-\ell_A(A/L_n) \\
& = & (b_0+b_1+\cdots+b_r)-(b_0+b_1+\cdots+b_{n-1}) \quad \text{by (1)}\\
& = & b_n+b_{n+1}+\cdots + b_r. 
\end{eqnarray*}
\par 
(2) Since $A/Q$ is Gorenstein, 
$\ell_A(\Hom_A(A/L_n, A/Q))=\ell_A(A/L_n)$ by Matlis duality. 
Thus 
\begin{eqnarray*}
\ell_A(A/(Q:L_n))
&=&\ell_A(A/Q)-\ell_A((Q:L_n)/Q) \\ 
&=&\ell_A(B)-\ell_A(\Hom_A(A/L_n, A/Q)) \\
&=&\ell_A(B)-\ell_A(A/L_n) \\
&=& b_n+b_{n+1}+\cdots +b_r, 
\end{eqnarray*}
as required.  
\end{proof}

\begin{lem} \label{L-Ineq}
For every $n=1,2,\ldots,r$, we have 
\begin{enumerate}
\item $L_n \subset Q: L_{r+1-n}$. 
\item $b_{r+1-n}+\cdots+b_{r-1}+b_r \le b_0+b_1+\cdots +b_{n-1}$. 
\item $\ell_A(A/L_n)+\ell_A(A/L_{r+1-n}) \ge e_0(I)$. 
\end{enumerate}
Furthermore, if equality holds in one of the above equations, then 
equality holds in the others. 
\end{lem}
\begin{proof}
(1) $\overline{I^n}\cdot \overline{I^{r+1-n}} \subset 
\overline{I^{n+(r+1-n)}}=\overline{I^{r+1}}\subset Q$ implies  
$\overline{I^n} \subset Q: \overline{I^{r+1-n}}$
$ =Q : L_{r+1-n}$. 
Hence $L_n \subset Q:L_{r+1-n}$.
\par \vspace{2mm} \par \noindent
(2) By (1), we have $\ell_A(A/(Q:L_{r+1-n})) \le \ell_A(A/L_n)$. 
Lemma $\ref{sum-B}$ yields 
\[
b_{r+1-n} + \cdots + b_{r-1}+b_r \le b_0+b_1+\cdots+b_{n-1}. 
\]
\par \vspace{2mm} \par \noindent 
(3) By Lemma $\ref{sum-B}$, we obtain 
\begin{eqnarray*}
\ell_A(A/L_n) &=& b_0+b_1+\cdots +b_{n-1}, \\
e_0(I)-\ell_A(A/L_{r+1-n})
&=& (b_0+b_1+\cdots+b_r)-(b_0+b_1+\cdots+b_{r-n}) \\
&=& b_{r+1-n}+\cdots+b_{r-1}+b_r. 
\end{eqnarray*}
Hence the inequality in (2) means 
$e_0(I)-\ell_A(A/L_{r+1-n}) \le \ell_A(A/L_n)$.
\end{proof}

\par \vspace{2mm}
 Although the following proposition is known (see \cite[Theorem 4.3]{HKU}), 
we will give a proof here for reader's convenience.

\begin{prop} \label{Mainpart}
The following conditions are equivalent$:$
\begin{enumerate}
\item $L_n=Q:L_{r+1-n}$ for every $n=1,2,\ldots, \lceil \frac{r}{2} \rceil$.  
\item[$(1)'$] $L_n=Q:L_{r+1-n}$ for every $n=1,2,\ldots, r$.
\item $b_{r+1-n}+\cdots +b_{r-1}+b_r=b_0+b_1+\cdots+b_{n-1}$ 
for every $n=1,2,\ldots, \lceil \frac{r}{2} \rceil$.  
\item[$(2)'$] $b_{r+1-n}+\cdots +b_{r-1}+b_r=b_0+b_1+\cdots+b_{n-1}$ 
for every $n=1,2,\ldots, r$. 
\item $\ell_A(A/L_n)+\ell_A(A/L_{r+1-n}) =e_0(I)$ 
 for every $n=1,2,\ldots, \lceil \frac{r}{2} \rceil$.  
\item[$(3)'$] $\ell_A(A/L_n)+\ell_A(A/L_{r+1-n}) =e_0(I)$ 
 for every $n=1,2,\ldots, r$. 
\item $B$ is Gorenstein. 
\end{enumerate}
\end{prop}

\begin{proof}
The equivalence between $(1)$, $(2)$ and $(3)$ (resp. $(1)'$, $(2)'$ and $(3)'$) follows from Lemma \ref{L-Ineq} (see also \cite[Theorem 4.3]{HKU}).
It suffices to show $(3) \Rightarrow (3)' \Rightarrow (4)  \Rightarrow (2)$. 
  
\par \vspace{2mm} \par \noindent 
$(3) \Longrightarrow (3)':$ 
One can easily check this. 
\par \vspace{2mm} \par \noindent 
$(3)' \Longrightarrow (4):$ 
By Lemma \ref{L-Ineq}, we may assume that 
$L_n=Q:L_{r+1-n}$ for every $n=1,2,\ldots, r$.  
Recall 
\[
B=\bigoplus_{n=0}^r B_n=A/I\oplus I/L_2\oplus L_2/L_3 \oplus 
\cdots \oplus L_r/Q. 
\]
It suffices to show $\dim_k \Soc(B)=1$.  
Let $x^{*}$ be a homogeneous element  in $\Soc(B)$. 
Namely, for $x \in \overline{I^n}$, we denote $x^{*}$ 
the image of $x$ in $B_n$. 
\par \vspace{1mm} \par \noindent
\underline{\bf The case where $x^{*} \in B_0$:}
\par \vspace{1mm}
By assumption, $L_r=Q:I$. 
Since $x^{*} B_r =0$ in $B_r$, we have
$x \in Q \colon L_r$\par\noindent
$=Q \colon (Q:I)=I$. 
Hence $x^{*}=0$ in $B_0$. 
\par \vspace{1mm} \par \noindent
\underline{\bf The case where $x^{*} \in B_n$ 
for some $n$ ($1 \le n \le r-1$):}
\par \vspace{1mm}
Since $x^{*}B_{r-n}=0$ in $B_r$, we have 
$x \in Q \colon  L_{r-n}=L_{n+1}$. 
Hence $x^{*}=0$ in $B_n$. 
\par \vspace{1mm} \par \noindent
\underline{\bf The case where $x^{*} \in B_r$:}
\par \vspace{1mm}
Since $x \in Q:\m$, we have $x^{*} \in \Soc((Q:I)/Q) 
=(Q\colon \m)/Q\cong k$. 
\par \vspace{1mm}\par \noindent
Therefore $\dim_k \Soc(B)=1$ and thus $B$ is Gorenstein. 

\par \vspace{2mm} \par \noindent 
$(4) \Longrightarrow (2):$ 
Since $B$ is Gorenstein, Matlis duality yields 
\[
b_k = b_{r-k} \quad \text{for every $k=0,1,\ldots, \lceil \frac{r}{2} \rceil$}
\]
This implies 
\[
b_0+b_1+\cdots + b_{n-1} = b_r+b_{r-1}+\cdots + b_{r+1-n}
\]
for each $n=1,2, \ldots, \lceil \frac{r}{2} \rceil$, as required. 
\end{proof}

\begin{proof}[Proof of Theorem $\ref{Main}$]
 Let $p=p_g(A)$ and $q_k=q(kI)$ for each $k=1,2,\ldots,r$. 
The equivalence between (1), (2) and (3) follows from Proposition 
\ref{Mainpart}. 
By Proposition \ref{q(nI)formula} and Lemma \ref{sum-B}, we have 
\[
\begin{array}{rcl}
b_2+b_3+\cdots +b_r=& \!\!\!\! \ell_A(\overline{I^2}/QI)  \!\!&= 
(p-q_1) - (q_1-q_2) \\
b_3+\cdots +b_r=& \!\!\!\!  \ell_A(\overline{I^3}/Q\overline{I^{2}}) 
 \!\!&= (q_1-q_2) - (q_2-q_3) \\
  & \vdots & \\
b_{r-1}+b_r=& \!\!\!\!  \ell_A(\overline{I^{r-1}}/Q\overline{I^{r-2}})  
\!\!&= (q_{r-3}-q_{r-2}) - (q_{r-2}-q_{r-1}) \\
b_r=& \!\!\!\!  \ell_A(\overline{I^{r}}/Q\overline{I^{r-1}})  \!\!&= 
 q_{r-2}-q_{r-1}.  \\
\end{array}
\]
\par 
By summing up, we get 
\[
b_2+2b_3+\cdots +(r-2)b_{r-1}+(r-1)b_r=p-q_1 
\]
Thus it follows from Kato's Riemann-Roch formula (Theorem \ref{RRformula}) that 
\begin{equation} \label{Eqbb}
b_2+2b_3+\cdots+(r-2)b_{r-1}+(r-1)b_r = \ell_A(A/I)-\chi(Z). 
\end{equation}
\par \vspace{2mm} \par \noindent 
$(1) \Longrightarrow (4)$:
By assumption, $B$ is Gorenstein. Hence $b_k=b_{r-k}$ for 
every $k=0,1,\ldots, r$. 
By subtracting $b_r=b_0=\ell_A(A/I)$ from both sides 
in Eq. (\ref{Eqbb}), we obtain  
\[
b_2+2b_3+\cdots +(r-2)b_{r-1}+(r-2)b_r= -\chi(Z), 
\]
that is, 
\[
2b_2+4b_3+\cdots +2(r-2)b_{r-1}+2(r-2)b_r= -2 \cdot \chi(Z)=Z^2+K_XZ. 
\]
When $r=2s-1$ (odd), 
\begin{eqnarray*}
\text{(LHS)} &=& \sum_{k=2}^{r-2} 2(k-1)b_k + 2(r-2)b_{r-1}+2(r-2)b_r \\
&=& \sum_{k=2}^{s-1} \left\{2(k-1)b_k+2(r-k-1)b_{r-k} \right\}
+2(r-2)b_{r-1}+2(r-2)b_r \\
&=& (r-2)(b_0+b_1+\cdots+b_r) \\
&=& (r-2)\ell_A(B) =(r-2)e_0(I)=-(r-2)Z^2. 
\end{eqnarray*}
Therefore $-(r-2)Z^2=Z^2+K_XZ$, that is, $(r-1)Z^2 +K_XZ=0$. 
\par \vspace{2mm}
When $r=2s$ (even), 
\begin{eqnarray*}
\text{(LHS)} &=& \sum_{k=2}^{r-2} 2(k-1)b_k + 2(r-2)b_{r-1}+2(r-2)b_r \\
&=& \sum_{k=2}^{s-1} \left\{2(k-1)b_k+2(r-k-1)b_{r-k} \right\}
+2(s-1)b_s+2(r-2)b_{r-1}+2(r-2)b_r \\
&=& (r-2)(b_0+b_1+\cdots+b_r) \\
&=& -(r-2)Z^2. 
\end{eqnarray*}
Similarly as above, we have $(r-1)Z^2 +K_XZ=0$. 
\par \vspace{2mm} \par \noindent 
$(4) \Longrightarrow (1)$: 
Note that $\ell_A(A/I)=b_0$. 
By assumption, $(r-1)Z^2+K_XZ=0$. 
This implies 
\[
\chi(Z)=\dfrac{r-2}{2}Z^2=-\dfrac{r-2}{2}(b_0+b_1+\cdots+b_r). 
\]
By Eq. ($\ref{Eqbb}$), we have 
\begin{equation} \label{Eqbbb}
b_2+2b_3+\cdots+(r-2)b_{r-1}+(r-1)b_r= b_0
+\dfrac{r-2}{2}(b_0+\cdots+b_r). 
\end{equation}
\par 
When $r=2s-1$ (odd), 
\begin{eqnarray*}
\text{(RHS)}-\text{(LHS)}
&=& \dfrac{1}{2} \left\{(b_0+b_1+\cdots +b_{s-1}) 
- (b_s+b_{s+1}+\cdots+b_r)  \right\} \\
& &+ \left\{(b_0+b_1+\cdots+b_{s-2})-(b_{s+1}+\cdots+b_r)\right\} \\
& &+ \cdots \cdots \\
& &+ \left\{(b_0+b_1)-(b_{r-1}+b_r) \right\} \\
& &+\left\{b_0-b_r\right\}. 
\end{eqnarray*}
Since all $\{~ \quad ~\} \ge 0$ 
by Lemma \ref{L-Ineq}, we have all $\{\quad \}$'s 
are equal to zero. 
Namely, 
\[
b_0+b_1+\cdots +b_{n-1} = b_{r+1-n}+\cdots + b_r 
\]
holds for every $n=1,2,\ldots, \lceil \frac{r}{2} \rceil=s$. 
\par \vspace{2mm}
When $r=2s$ (even),  
\begin{eqnarray*}
\text{(RHS)}-\text{(LHS)}
&=& \left\{(b_0+b_1+\cdots +b_{s-1}) 
- (b_{s+1}+b_{s+2}+\cdots+b_r)  \right\} \\
& &+ \left\{(b_0+b_1+\cdots+b_{s-2})-(b_{s+2}+\cdots+b_r)\right\} \\
& &+ \cdots \cdots \\
& &+ \left\{(b_0+b_1)-(b_{r-1}+b_r) \right\} \\
& &+\left\{b_0-b_r\right\}. 
\end{eqnarray*}
\par 
By a similar argument as above, we obtain 
\[
b_0+b_1+\cdots +b_{n-1} = b_{r+1-n}+\cdots + b_r 
\]
holds for every $n=1,2,\ldots, \lceil \frac{r}{2} \rceil=s$. 
\par \vspace{2mm}
By Proposition \ref{Mainpart}, $B$ is Gorenstein and thus 
$\overline{G}(I)$ is Gorenstein. 

\par \vspace{2mm}
In what follows, we suppose that $\overline{G}(I)$ is Gorenstein. 
As seen in the first part of the proof,
\[
\ell_A(A/I)=b_0=b_r
=\ell_A(\overline{I^{r}}/Q \overline{I^{r-1}})
=q_{r-2}-q_{r-1} \le q_{r-2}. 
\]
Since 
$p_g(A)=q_0 > q_1 > \cdots > q_{r-2}$ is a strict decreasing sequence
of integers, we have 
\[
\ell_A(A/I) \le  q_{r-2} \le q_{r-3}-1 \le \cdots \le p_g(A)+2-r,
\]
as required. 
\end{proof}

\par \vspace{2mm}
Using Theorem \ref{Main}, we can prove the following fact:
If $A$ has an ideal $I=I_Z$ for which $\overline{G}(I)$ is Gorenstein 
with $\br(I) =r \ge 2$, then it admits an elliptic ideal $J$ 
for which $\overline{G}(I)$ is Gorenstein. 

\begin{cor} \label{VerNTC}
Let $(A,\m)$ be as in Theorem $\ref{Main}$. 
Assume that $\overline{G}(I)$ is Gorenstein for some $I=I_Z$ with 
$\br(I)=r \ge 2$. 
Put $I_k=\overline{I^k}$ for each integer $k \ge 1$. 
If $k$ is a  divisor of $r-1$, then 
\begin{enumerate}
\item $\br(I_k)=1+\frac{r-1}{k}$. 
\item $\overline{G}(I_k)$ is Gorenstein. 
\end{enumerate}
In particular, $I_{r-1}$ is an elliptic ideal such that 
$\overline{G}(I_{r-1})$ is Gorenstein.  
\end{cor}

\begin{proof}
Since $\overline{G}(I) \cong \overline{\mathcal{R}'}(I)/t^{-1} \overline{\mathcal{R}'}(I)$ is Gorenstein, 
so is $\overline{\mathcal{R}'}(I)$. 
As $ \overline{\mathcal{R}'}(I_k)$ is 
a direct summand of $ \overline{\mathcal{R}'}(I)$, 
it is Cohen-Macaulay. 
Thus $\overline{G}(I_k)$ is also Cohen-Macaulay and $I_k=I_{kZ}$
and (1) follows from  Corollary \ref{br(Ik)}. 
Also, since $(r-1)Z^2 + K_XZ = 0$, by \ref{Main}, (2) also follows from 
Theorem \ref{Main}.
\end{proof}

\bigskip
\section{Normal Tangent cone for the maximal ideal of a Brieskorn hypersurface}

\par \vspace{2mm}
Let  $2 \le a \le b \le c$ be integers, and let $K$ be an 
algebraically closed field of characteristic $p \ge 0$, which 
$p$ does not divide $abc$ or $p=0$. 
Then 
\[
A=K[x,y,z]_{(x,y,z)}/(x^a+y^b+z^c)
\] 
is called a \textit{Brieskorn hypersurface} of type $(a,b,c)$. 
Put $\m=(x,y,z)A$ and $\widehat{A}$ denotes the $\m$-adic 
completion of $A$: 
$\widehat{A}=K[[x,y,z]]/(x^a+y^b+z^c)$. 
\par \vspace{2mm}
Moreover, we put 
\[
d=\mathrm{gcd}(a,b),  
\quad 
n_k=\lfloor \frac{kb}{a} \rfloor \;\text{for $k=1,2,\ldots,a-1$}. 
\]

\par \vspace{2mm}
In what follows, let $(A,\m)$ be a Brieskorn hypersurface 
of type $(a,b,c)$. 
Then $A$ is an excellent Gorenstein normal local domain of $\dim A=2$,  
and $\m$ is an $\m$-primary integrally closed ideal with minimal 
reduction $Q=(y,z)A$. 
 We assume that $\m$ is represented by a cycle $Z$ on a resolution $X$.

\par \vspace{2mm}
The main aim of this section is to give a complete criterion for Gorensteinness of $\overline{G}(\m)$ as an application of Theorem \ref{Main}. 
In fact, we prove the following theorem. 

\begin{thm} \label{BmaxThm}
Let $A=k[[x,y,z]]/(x^a+y^b+z^c)$ and put $r=n_{a-1}=\lfloor \frac{(a-1)b}{a} \rfloor$ and $d=\gcd(a,b)$. 
Then the following conditions are equivalent$:$
\begin{enumerate}
\item $\overline{G}(\m)$ is Gorenstein. 
\item $\ell_A(A/L_n)+\ell_A(A/L_{r+1-n})=a$ for every$n=1,2,\ldots, r$.  
\item $(r-1)Z^2+K_XZ=0$.  
\item $b \equiv  0$ or $d \pmod{a}$. 
\end{enumerate}
\end{thm}

\par \vspace{2mm}
In order to prove Theorem \ref{BmaxThm}, we recall the following result. 

\begin{lem}[\textrm{cf. \cite{OWY4}}] \label{BmaxCM}
Let $(A,\m)$ be as in Theorem $\ref{BmaxThm}$. Then  
\begin{enumerate}
\item $\overline{\m^n}=Q^n + xQ^{n-n_1}+x^2 Q^{n-n_2}+\cdots +x^{a-1}Q^{n-n_{a-1}}$. 
\item $\overline{G}(\m)$ is Cohen-Macaulay and $e_0(\overline{G}(\m))=e_0(\m)=a$. 
\item $\br(\m)=\nr(\m)=n_{a-1}$.  
\end{enumerate}
\par \vspace{1mm}
In particular, if $b=an_1+\delta$, $1 \le \delta \le a-1$, then 
$ r:=\br(\m)  =(a-1)n_1+\delta-1 = b-n_1-1$. 
\end{lem}

\begin{cor} \label{Lncal}
For an integer with $1 \le n \le r$, we have 
\[
\ell_A(A/L_n)=\min\{k \in \bbZ_{+} \,|\, n \le n_k \}.
\]
We put $\ell_n=\ell_A(A/L_n)$ for $n=1,2,\ldots,r$. 
\end{cor}

\begin{proof}
Suppose  $n_{k-1} < n  \le n_k$. 
Then Lemma \ref{BmaxCM}(1) yields $L_n=Q+(x^k)$. 
Thus $\ell_A(A/L_n)=\ell_A(K[x]/(x^k))=k$. 
\end{proof}

\begin{proof}[Proof of Theorem $\ref{BmaxThm}$]
The equivalence of $(1)$, $(2)$ and $(3)$ follows from 
Theorem $\ref{Main}$. 
\par \vspace{2mm} \par \noindent 
$(3) \Longrightarrow (4):$
By \cite[Proposition 3.8(1)]{OWY4}, we have 
\[
\ell_A(\overline{\m^{k+1}}/Q\overline{\m^k}) 
= a - \left\lceil \dfrac{a(k+1)}{b} \right\rceil 
\]
for $k=0,1,\ldots,r$. 
Then 
\[
\sum_{k=1}^r \ell_A(\overline{\m^{k+1}}/Q\overline{\m^k}) 
= ar - \sum_{k=1}^r \left\lceil \dfrac{a(k+1)}{b} \right\rceil.  
\]
One can easily see that 
\[
\text{(RHS)}=\dfrac{d+b(a-1)-3a+2}{2}. 
\]
On the other hand, 
 the argument of the proof of \thmref{Main} and 
Kato's Riemann-Roch formula yields 
\[
\text{(LHS)}=p_g(A)-q(\m)=1-\chi(Z)=1+\dfrac{Z^2+KZ}{2}. 
\]
Since $Z^2=-a$, we have $KZ=d+ab-b-2a$. 
\par \vspace{1mm}
Hence the assumption $(3)$ means 
\[
(r-1)(-a)+d+ab-b-2a=0. 
\]
Then $b \equiv d \pmod{a}$. 
\par \vspace{2mm} \par \noindent 
$(4) \Longrightarrow (3):$
\par
First suppose $d=a$, that is, $a$ is a divisor of $b$. 
Then $n_k=kb/a$ for every $k=1,2,\ldots,a-1$ and 
$r=n_{a-1}=(a-1)n_1$. 
Then $\overline{G}(\m)$ is a hypersurface, and thus Gorenstein. In fact, 
\begin{eqnarray*}
\overline{\mathcal{R}'}(\m) 
& \cong &  K[X,Y,Z,U]/(X^a+Y^b+Z^cU^{c-b}) \\[2mm]
&  & \text{where $X=xt^{n_1}$, $Y=yt$, $Z=zt$ and $U=t^{-1}$, and}  \\[2mm] 
\overline{G}(\m) & \cong &
\left\{ 
\begin{array}{ll}
K[X,Y,Z]/(X^a+Y^b+Z^b) & \;\text{$b=c$;} \\
K[X,Y,Z]/(X^a+Y^b) & \;\text{$b<c$}. 
\end{array}
\right.
\end{eqnarray*}

\par \vspace{2mm}
Next we consider the case of $d=1$ and $b \equiv 1 \pmod{a}$. 
For every $k$ with $1 \le k \le a-1$, we have 
\[
n_k
=\lfloor \frac{kb}{a} \rfloor 
= \lfloor \frac{k(n_1a+1)}{a} \rfloor
=kn_1+ \lfloor \dfrac{k}{a} \rfloor = kn_1. 
\]
If we put $Y=yt$, $Z=zt$, $X=xt^{n_1}$ and $U=t^{-1}$,  
then since 
\[
X^a=(xt^{n_1})^a=(-y^b-z^c)t^{n_1a} 
=-y^bt^{b-1}-z^c t^{b-1} 
= -Y^bU-Z^cU^{c-b+1} 
\]
we have 
\begin{eqnarray*}
\overline{\mathcal{R}'}(\m) &=& 
 K[xt, yt, zt, \{x^kt^{n_k}\}_{2 \le k \le a-1}, t^{-1}] \\
&=&  K[yt,zt,xt^{n_1},t^{-1}] \\
& \cong &  K[X,Y,Z,U]/(X^a+Y^bU+Z^cU^{c-b+1}). 
\end{eqnarray*}
Thus $\overline{G}(\m) \cong  K[X,Y,Z]/(X^a)$ is 
a hypersurface, and thus Gorenstein. 
\par \vspace{2mm} 
Finally, we consider the case of $a>  d=\gcd(a,b)>1$ and 
$b \equiv d \pmod{a}$. 
Put $a'=a/d$ and $b'=b/d$. Then $b'=n_1a'+1$.  
One can easily see that $n_k=kn_1$ for $1 \le k \le a'-1$, 
and $n_{a'}=b'$. 
For any integer $p$ $(1\le p \le d-1)$, if $pa'+1 \le k \le pa'+a'-1$, 
then 
\begin{eqnarray*}
n_{k} &=& \lfloor \frac{k(n_1a'+1)}{a'} \rfloor 
= kn_1+ \lfloor \frac{k}{a'} \rfloor = kn_1+ \lfloor \frac{pa'+(k-pa')}{a'} \rfloor \\[1mm]
&=& kn_1+p = pb'+(k-pa')n_1. 
\end{eqnarray*}
Hence $(k,n_k)=p(a',b')+(k-pa')(1,n_1)$. 
Similarly, if $k=pa'$ $(2 \le p \le d-1)$, then $(k,n_k)=p(a',b')$. 
Therefore
\begin{eqnarray*}
\overline{\mathcal{R}'}(\m) 
&=& K[xt,yt,zt,xt^{n_1},x^{a'}t^{b'},t^{-1}] \\[1mm]
& \cong & K[X,Y,Z,W,U]/(X^{a'}-WU, W^d+Y^b+Z^cU^{c-b}),
\end{eqnarray*} 
where $X=xt^{n_1}$, $Y=yt$, $Z=zt$, $W=x^{a'}t^{b'}$, and $U=t^{-1}$. 
Indeed, 
\[
W^d=(x^{a'}t^{b'})^d = x^at^b=-(y^b+z^c)t^b=-Y^b-Z^cU^{c-b}, 
\]
and 
\[
 X^{a'}=x^{a'}t^{n_1a'}=x^{a'}t^{b'-1}=(x^{a'}t^{b'})(t^{-1})=WU.
\]
Moreover, 
\begin{eqnarray*}
\overline{G}(\m) 
& \cong &
\left\{
\begin{array}{ll}
K[X,Y,Z,W]/(X^{a'}, W^d+Y^b+Z^c), &\, \text{if $b=c$}; \\[1mm]  
K[X,Y,Z,W]/(X^{a'}, W^d+Y^b), & \, \text{if $b<c$}. 
\end{array}
\right.
\end{eqnarray*}
Thus $\overline{G}(\m)$ is a complete intersection, and thus Gorenstein. 
\end{proof}

\begin{cor} \label{cor: Bries-ex}
We use the same notation as in Theorem $\ref{BmaxThm}$. 
\begin{enumerate}
\item $\br(\m) \ge a-1=r(\m)$.  
\item If both $r$ and $a$ are odd, 
then  $\overline{G}(\m)$ is not Gorenstein. 
\item Let $s \ge 1$ be a positive integer. 
If $(a,b,c)=(3,3s,3s)$, then  
$\overline{G}(\m)$ is Gorenstein and $\br(\m)=2s$.  
\item If $a=4$ and $\overline{G}(\m)$ is Gorenstein, then 
$\br(\m) \equiv 0,1 \pmod{3}$. 
\end{enumerate}
\end{cor}

\begin{proof}
(1) It follows from $r = \lfloor \frac{(a-1)b}{a} \rfloor \ge a-1$. 
 \par \vspace{2mm} \par \noindent 
(2) Suppose that $\ol G(\m)$ is Gorenstein. 
Then $(r-1)Z^2 + K_XZ=0$. This implies $\frac{(r-2)a}{2}=-\chi(Z)$ is an integer. 
\par \vspace{2mm} \par \noindent 
(3)  It follows from Theorem \ref{BmaxThm}. 
\par \vspace{2mm} \par \noindent 
(4) It follows from Theorem \ref{BmaxThm} that $b=4m$, $4m+1$, or 
$4m+2$. 
Then $r=3m$. $3m$, $3m+1$, respectively. 
\end{proof}

\begin{exam}
Let $A=K[x,y,z]/(x^3+y^5+z^5)$, where $K$ is an algebraically closed field of characteristic $p \ne 3,5$. Then 
\begin{enumerate}
\item $a=3$, $b=5$, $n_1=1$ and $r=n_2=3$. 
\item $\overline{\m^2}=\m^2$, $\overline{\m^3}=\m^3+(x^2)$, 
and $\overline{\m^4}=(y,z)\overline{\m^3}$. 
\item If we put $X=xt$, $Y=yt$, $Z=zt$ and $W=x^2t^3$, and $U=t^{-1}$, 
then 
\begin{eqnarray*}
\overline{\mathcal{R}'}(\m) \!\!\!\!
&=& \!\!\!\! K[xt,yt,zt,x^2t^3,t^{-1}] \\[1mm]
&\cong &  \!\!\!\! 
K[X,Y,Z,W,U]/(X^2-WU, XW+Y^5U+Z^5U, W^2+XY^3+XZ^5). 
\end{eqnarray*}
In particular, 
\[
\overline{G}(\m) 
\cong K[X,Y,Z,W,U]/(X^2, XW, W^2+XY^5+XZ^5). 
\]
\end{enumerate} 
\end{exam}

\begin{exam} [cf. {\cite{OWY5}}]  
 Let $a \ge 2$ be an integer, and 
let $A=\mathbb{C}[[x,y,z]]/(x^a+y^a+z^a)$ be a Brieskorn hypersurface. 
Put $\m=(x,y,z)A$.  
\begin{enumerate}
\item $\overline{G}(\m)$ is Gorenstein and $\br(\m)=a-1$. 
\item If $a=2$, then $A$ is a rational singularity. 
Thus $\m$ is a good $p_g$-ideal and $\overline{G}(\m)$ is Gorenstein. 
Moreover, $A$ does not have any elliptic ideal. 
\item If $a=3$, then $A$ is an elliptic singularity and 
$\m$ is an elliptic ideal such that $\overline{G}(\m)$ is Gorenstein. 
\item If $a \ge 4$, then $\br(\m)=a-1 \ge 3$ and thus 
$\m$ is neither a $p_g$-ideal nor an elliptic ideal but 
$\overline{G}(\m)$ is Gorenstein.  
On the other hand, $\m^{a-1}$ is an elliptic ideal such that 
$\overline{G}(\m^{a-1})$ is Gorenstein. 
\end{enumerate}
\end{exam}


\bigskip
\section{The  finiteness of the graphs of ideals $I$ with Gorenstein $\overline{G}(I)$}

Throughout this section, let $(A,\m)$ be as in Section 1. 

\begin{defn} \label{ArithGenus}
For a cycle $C>0$, let $p_a(C):=1-\chi(C)$.
We define the {\em arithmetic genus} of $A$ by $p_a(A)=\max\defset{p_a(C)}{C>0}$; this is independent of the choice of the resolution.
We put $\chi(A) = \min\defset{\chi(C)}{C>0}$, namely, $\chi(A)=1-p_a(A)$.
\end{defn}

\par \vspace{1mm}
In this section, we introduce the weighted dual graph of 
an $\m$-primary ideal $I=I_Z$ 
of $A$ and prove the finiteness of 
the weighted dual graphs of $I$ with Gorenstein normal
 tangent cone with $\br(I_Z)\ge 2$.
\par 
Let $X \to \spec (A)$ be a resolution of singularities
 with exceptional set $E$.
Let $E=\bigcup_{i=1}^n E_i$ be the decomposition into the irreducible 
components.
For $h\in \m$, we denote by $(h)_E$ the exceptional part of $\di_X(h)$.
Let $E_j^* \in \sum_{i=1}^n \Q E_i$ denote the element which satisfies $E_j^* E_i = -\delta_{ij}$.
 Let $h^i( \; * \; ):= \ell_A( H^i( \; * \; ))$. 

\begin{defn}\label{d:graph}
Let $X_I \to \spec (A)$ be the normalization of the blow-up of $I$, and $X\to X_I$ the resolution which is minimal with the property that the exceptional set $E \subset X$ is a simple normal crossing divisor.
We call $X \to \spec (A)$ the {\em minimal log resolution} of $I$.
Assume that $I$ is represented by a cycle $Z$ on $X$ and let $Z=\sum z_i E_i^*$.
We define the {\em weighted dual graph} $\Gamma(I)$ of $I$ as follows.
Let $\Gamma_E$ denote the usual weighted dual graph of $E$.
Let $\cE^*$ denote the set of vertices of $\Gamma_E$ corresponding to $E_i$ with $z_i\ne 0$.
For each vertex $v_i\in \cE^*$, corresponding to $E_i$, we add a new edge $e_i$ and a new vertex $w_i$ so that $v_i$ and $w_i$ are connected by $e_i$, and put a weight $z_i$ on $w_i$.
Then we obtain the graph $\Gamma(I)$ as the sum of $\Gamma_E$ and those new edges and vertices with weights $z_i$.
When we illustrate the dual graph $\Gamma(I)$, the vertices $w_i$ are denoted by arrowheads and its weight by $(z_i)$;
the genus is omitted if it is zero as usual.
\end{defn}

\begin{rem}
$\Gamma(I)$ is uniquely determined by $I$.
We can recover the coefficients of $Z$ from $\Gamma(I)$.
\end{rem}

\begin{ex}  \label{ex:236}
Assume that $A$ is a hypersurface singularity defined by the polynomial $x^2+y^3+z^6$.
Then we have the following.
\begin{center}
\begin{picture}(120,35)(80,10)
\put(80,25){\makebox(0,0){$\Gamma(\m)$: }}
\multiput(115,25)(40,0){2}{\circle*{4}}
\put(115,25){\line(1,0){40}}
\put(155,25){\vector(1,0){40}}
\put(115,15){\makebox(0,0){$[1]$}}
\put(195,15){\makebox(0,0){$(1)$}}
\put(115,35){\makebox(0,0){$-2$}}
\put(155,35){\makebox(0,0){$-1$}}
\end{picture}
\hspace{2cm}
\begin{picture}(120,35)(80,10)
\put(80,25){\makebox(0,0){$\Gamma(\m^2)$: }}
\put(115,25){\circle*{4}}
\put(115,25){\vector(1,0){40}}
\put(115,15){\makebox(0,0){$[1]$}}
\put(155,15){\makebox(0,0){$(2)$}}
\put(115,35){\makebox(0,0){$-1$}}
\end{picture}
\end{center}
\end{ex}

\begin{defn}
For $r \in \Z_{\ge 1}$, let $\G(A,r)$ denote the set of $\m$-primary integrally closed ideals $I$ of $A$ such that $\br(I)=r$ and $\overline{G}(I)$ is Gorenstein, and let $\GG(A,r)=\defset{\Gamma(I)}{I \in \G(A,r)}$.
\end{defn}

\begin{thm}\label{t:GG}
$\GG(A,r)$ is a finite set if $r \ge 2$.
Moreover, the restriction of the map $\Gamma\:  \G(A,r) \to \GG(A,r)$ to the set 
\[
\G(A,r)_0:=\defset{I \in \G(A,r)}{\text{$I$ is represented on the minimal resolution}}
\]
 is injective.
In particular, $\Gamma$ is injective on the set
\[
\defset{I \in \G(A,r)}{I = I_Z, \; \chi(Z)=\chi(A)}.
\]
\end{thm}

Note that if $A$ is elliptic, then every $I \in \G(A,2)$  is represented on the minimal resolution.
To prove \thmref{t:GG}, we use \thmref{Main} and some results.

\begin{lem}\label{l:minres}
 For any cycle  $W>0$,
 the following is a finite set.
\[
\defset{Z}{\text{$Z$ is an anti-nef cycle  on $X$ and 
$Z - W \not \ge 0$}}
\]
\end{lem}

\begin{proof}
It follows from the fact that any anti-nef cycle on $X$ is a linear combination of $\{E_j^*\}_j$
 with coefficients in $\Z_{\ge 0}$ and for any $E_j^*$ and any irreducible component 
 $E_i$ of $E$, the coefficient of $E_i$ of $E_j^*$ is positive. 
\end{proof}

\par \vspace{1mm}
Recall that for cycles $C,D>0$, we have 
$\chi(C+D)=\chi(C)+\chi(D)-CD$.

\begin{prop}\label{p:ZcA}
Assume that  $X \to \spec (A)$ is the minimal resolution.
Let $Z_K$ denote the cycle such that $(K_X+Z_K)E_i=0$ for all $i$, 
and let  $Z>0$ be a cycle on $X$ such that $\cO_X(-Z)$
has no fixed component and  $\chi(Z) \le 0$.
Then $Z\not > Z_K$. 
Therefore, the set of such cycles  $Z$  is a finite set.
\end{prop}

\begin{proof}
 Note that $\ell_A(A/H^0(\cO_X( -Z_K ) ) )=p_g(A)$ and that  
$\ell_A(A/H^0(\cO_X( -Z ) ) )=\chi(Z)- h^1(\cO_X( -Z ) ) ) + p_g(A)$ by Kato's Riemann-Roch formula.
Assume that $Z = Z_K +Y$ with $Y > 0$. Then, since 
\[
\ell_A(A/H^0(\cO_X( -Z ) ))  
\ge \ell_A(A/H^0(\cO_X( -Z_K ) ) )=p_g(A),
\] 
we have  $\chi(Z)= h^1(\cO_X( -Z ) ) )=0$.  
Since 
\[
H^0(X, \cO_X(-Z_K)) 
= H^0(X, \cO_X( -Z)),
\]
 $Y$ is the fixed component of $\cO_X(-Z_K)$. 
Hence $\chi(Y)>0$ by \cite{K}.
On the other hand,  
\[
0 = \chi(Z) = \chi(Z_K) + \chi(Y) - Y Z_K = \chi(Y) - Y Z_K. 
\]
The minimality of the resolution $X$ implies that $Z_K$ is anti-nef,
and thus  $\chi(Y) = Y Z_K \le 0$; however, it contradicts that $\chi(Y) > 0$. 
Hence $Y\not>0$. 
The last assertion follows from \lemref{l:minres}
\end{proof}

\begin{prop}\label{p:repMin}
Assume that $I$ is minimally represented by a cycle $Z>0$ on $X$, 
 namely, $ZC<0$ holds for every $(-1)$-curve $C$ on $X$. 
Let  $X_0 \to \spec A$ be the minimal resolution and $\phi: X\to X_0$ the natural morphism. 
Assume that $\phi$ is the composite $\phi_1\circ \cdots \circ \phi_m$ of blow-ups $\phi_i$.
Let $F_i$ denote the total transform of the exceptional divisor of $\phi_i$ on $X$ and $Z_0=\phi_*Z$.
If $F=\sum_{i=1}^m a_i F_i=Z - \phi^*Z_0$, then 
 we have the following formulas: 
\begin{enumerate}
\item 
 $\chi(Z)-\chi(Z_0) =\chi(F) = \displaystyle{\sum_{i=1}^m} a_i(a_i+1)/2 
 \le \chi(Z)-\chi(A)$. 
\par \vspace{1mm} \par \noindent 
In particular, if $\chi(Z) = \chi(Z_0)$, we must have $X = X_0$. 
\item $- Z^2 = - Z_0^2 + \displaystyle{\sum_{i=1}^m} a_i^2$. 
\end{enumerate}
\end{prop}

\begin{proof}
 First note that $F_iF_j=-\delta_{ij}$.
 Hence (2) follows.
If $m>0$ and $a_i=0$ for some $i$, then $ZF_i=0$, and hence 
there exists a $(-1)$-curve $C$ such that $ZC=0$; it contradicts the minimality of the representation.
Assume that $a_i>0$ for all $i$.
We have $\chi(Z)=\chi(\phi^*Z_0)+\chi(F)=\chi(Z_0)+\chi(F)$. 
Since 
$\chi(a_iF_i)=\chi((a_i-1)F_i)+\chi(F_i) -(a_i-1)F_i^2$, 
we have \[
\chi(F)=\sum_{i=1}^m \chi(a_i F_i)=\sum_{i=1}^m a_i(a_i+1)/2.
\]
The inequality follows from the definition of $\chi(A)$.
\end{proof}

\begin{rem}\label{r:ai}
In the situation of \proref{p:repMin}, take a general element $h\in I$ and let $V=\di_X(h)-(h)_E$ and $V_0=\phi_*V$.
Let $p_i$ be the center of the blow-up $\phi_i\: X_i\to X_{i-1}$, where $X=X_m$, and $V_i$ the proper transform of $V_0$ on $X_{i-1}$.
 Then $Z=(h)_E$ and $\di_{X_0}(h)=Z_0+V_0$, where $Z_0= \phi_*Z$ as above.
Therefore, $\phi^*V_0=\di_X(h)-\phi^*Z_0=F+V$, because $Z=\phi^*Z_0+F$.
This implies that the multiplicity of $V_{i-1}$ at $p_i$ is $a_i$.
Moreover, $p_i$ is a base point of $\cO_{X_{i-1}}(-Z_i)$, where $Z_i = (\phi_i\circ \cdots \circ \phi_m)_*Z$.
\end{rem}

\begin{proof}[Proof of $\thmref{t:GG}$]
It immediately follows from \proref{p:repMin} that the ideals $I_Z$  with $\chi(Z)=\chi(A)$ are represented on the minimal resolution.
\par 
Let us consider the situation as in \proref{p:repMin}.
Note that by \thmref{Main}, $\chi(Z)\le 0$ for $I_Z \in \G(A,r)$ with $r\ge 2$.
We will show the finiteness of the process obtaining the minimal log resolution from $X_0$.
We have the finiteness of the cycles $Z_0=\phi_*Z$  by \proref{p:ZcA}.
In particular,  $\G(A,r)_0$ is a finite set.
Since what we have to know is the combinatorial data, we may regard the number of choices of the center of blow-ups $\phi_i$ as finite; 
 if $E^{(i-1)}\subset X_{i-1}$ is the exceptional set, then combinatorially, there is no need to distinguish between two smooth points of $E^{(i-1)}$ if they are on the same irreducible component of $E^{(i-1)}$. 
Moreover, the combinatorial type of  those plane curve singularities  are determined by the weighted dual graph $\Gamma$ of $A$, and the intersection numbers 
 $-ZE_j=-(F+\phi^*Z_0)E_j$, which is $z_j$ in \defref{d:graph},
  are determined by $Z_0(\phi_*E_j)$ and $\{a_i\}_i$ and $\{F_iE_j\}_i$. 
Note that if $h\in I$ is a general element and $\di_X(h)=Z+V$, then $ZE_j=-VE_j$, and any point of $V \cap E$ is a non-singular point of $E$ because $\cO_X(-Z)$ is generated.
Since the minimal log resolution $\tilde X \to \spec (A)$ of $I$ is obtained by resolving non-normal crossing points of $E \subset X$, the data of the arrowheads of $\Gamma(I)$ is already determined by the intersection numbers $ZE_j$ and $\Gamma$.
Hence $\GG(A)$ is a finite set  and $\Gamma |_{\G(A,r)_0}$ is injective. 
\end{proof}

\begin{rem}\label{r:pginfty}
If $I$ is a $p_g$-ideal, then $\ol G(I)$ is Gorenstein if and only if $I$ is good (see \cite{OWY6}).
Therefore, the set $\GG(A,1)$ is infinite because the set of $p_g$-ideals of $A$ is closed under product and if $I$ is a good ideal of $A$, then so is $I^n$ for every $n\in \Z_{\ge 1}$. 
\end{rem}

\section{Homogeneous hypersurfaces of degree $4$ and $5$}

\begin{noname}\label{5;setting}  In this section let  $A = k[X,Y,Z]_{(X,Y,Z)}/ (f)$, 
where $f$ is a homogeneous polynomial of degree $d\ge 3$.
The aim of this section is to classify the $\m$-primary integrally closed ideals
 $I \subset A$ such that 
$\br(I) \ge 2$ and  $\overline{G}(I)$ is Gorenstein in the case $d\le 5$. 
\end{noname}

 Surprisingly, there are very few of such ideals and we can 
list all of them. 

\begin{thm}\label{Main: d-4,5} Let $A, I$ be as in \ref{5;setting}. 
If $\br(I) \ge 2$ and  $\overline{G}(I)$ is Gorenstein, then $I$ is 
one of the followings.
\begin{enumerate}
\item If $d=3$, then $I=\m$, the maximal ideal. 
\item If $d=4$, then $I=\m$ or $\m^2$.
\item If $d=5$, then $I= \m$, $\m^3$ or $I(L) = (L) +\m^2$, which is defined in Definition 
 \ref{I(L)def}.    
\end{enumerate}
\end{thm}
\par 
Let us prove the Theorem \ref{Main: d-4,5} in several steps. 
\par 
We basically use the notation from the preceding section. 
Let $X_0$ be the minimal resolution of $\Spec(A)$ and 
$C$ denotes the unique exceptional curve  of  $X_0$. 
Let $I = I_Z$ be minimally represented by a cycle $Z$ on a resolution $X$ and 
$\phi : X \to X_0$ be the natural morphism.  
We denote  $Z_0 = \phi_*(Z) = uC$ and let $E_0$ be the proper transform of $C$ on $X$.
Also, following Proposition \ref{p:repMin}, we denote 

\[Z = \phi^*(uC) + F, \quad F = \sum_{i=1}^m a_iF_i.\] 

We put $S(t)=t(t-1)(t-2)/6$. 
\par 

First, let us recall the fundamental results. 

\begin{lem}\label{deg d hyp} 
We have the following.
\begin{enumerate}
\item  $K_{X_0}=(2-d)C$, $g(C) =  (d-1)(d-2)/2$, and $p_g(A)=S(d)$.
\item $\chi(Z_0)  = \chi(uC)  = du (u - d +2 )/2$. Hence $\chi(Z_0) \le 0$ if and only if $u\le d - 2$.
\item $\br(I) \le \br(\m) = d-1$. 
\item For $0\le n \le d$, $\ell_A(A/\m^n) = S(n+2), q(n\m) = S(d-n)$.
\item $I \subset \m^u$, and  $I=\m^u$ if and only if $\ell_A(A/I)=S(u+2)$.
Moreover, 
\[
\ell_A(A/I)=\ell_A(A/\m^u)+\chi(F)-q(I)+q(u\m), \quad q(I)\ge q(u\m).
\]
\end{enumerate}
\end{lem}

\begin{proof} 
The claims (1) and (2) are clear.
The claims (3) and (4) follow from the formulas in \cite[\S 4.1]{OWY5}. 
\par
Now we prove (5).
 We have the following exact sequence:
\[
0 \to \cO_X(-Z) \to \cO_X(-\phi^*Z_0) \to \cO_F \to 0.
\]
Since $H^0(\cO_X(-\phi^*Z_0))=H^0(\cO_{X_0}(-Z_0))=\m^u$, 
we have $I \subset \m^u$ and 
\[
\ell_A(A/I)-\ell_A(A/\m^u)= \ell_A(\m^u/I)=\chi(F) - q(I) + q(u \m).
\]
From (4), $\ell_A(A/I)=\ell_A(A/\m^u)$ if and only if $\ell_A(A/I)=S(u+2)$.
Since $H^1(\cO_F)=0$, we have $q(I) \ge q(\m^u)$.
\end{proof}

\par 
In the following, we write $r=\br(I)$ and $g=g(C)$.

\begin{lem}\label{l:u} 
Assume that $\overline{G}(I)$ is Gorenstein with 
$r \ge 2$. Then  $\chi(Z_0) \le \chi(Z)\le 0$ and $u\le d-2$.
\end{lem}

\begin{proof} Since $\chi(Z)=(r-2)Z^2/2 \le 0$ if $\ol G (I)$ is Gorenstein by \thmref{Main}, it follows from  \proref{p:repMin} (1) and Lemma \ref{deg d hyp} (2).
\end{proof}

The invariant $q(\infty I)$ plays an important role in our proof of 
Theorem \ref{Main: d-4,5}.

\begin{lem}\label{l:q(inf)}   $q(\infty I ) = 0$ if $Z E_0 < 0$ and $q(\infty I )= q((r-1) I) \ge g$ 
if  $Z E_0=0$. 
\end{lem}
\begin{proof}
It follows from \proref{q(nI)formula} and \cite[3.5, 3.6]{OWY5}.
\end{proof}

\par 
The number $Z E_0$ plays an important role in our proof of  \thmref{Main: d-4,5}.
We recall the following fact.  

\begin{lem}\label{ZE_0}
 Let $Z = \phi^*(u C) + F = \phi^*( u C) + \sum_{i=1}^m a_i F_i$ as 
in Proposition \ref{p:repMin} and we define $P$ be a subset of $\{1, \ldots , m\}$ 
defined by the condition that $i\in  P$ if the blowing-up $\phi_i\: X_i \to X_{i-1}$
 occurs on the proper transform of $C$ on $X_{i-1}$. Then we have  
\[Z E_0 = - ud + \sum_{i\in P} a_i.\]
\end{lem}

\begin{proof}
By the definition of $F_i$ and $P$, we have $F_iE_0= 0$ if $i \not\in P$ and $F_iE_0= 1$ if $i \in P$.
Hence 
\[
ZE_0 = \phi^*(u C)\phi^*C + FE_0 = - ud + \sum_{i\in P} a_i.
\qedhere
\]
\end{proof}

\par 
We know that $\overline{G}(\m)= G(\m) = k[X,Y,Z]/ (f)$ is Gorenstein and also 
$\overline{G}(\m^n)$ is Cohen-Macaulay for every $n\ge 1$.  

\begin{ex} [cf. \corref{VerNTC}, \cite{GI}] \label{e:m^n}
$\overline{G}(\m^n)$ is Gorenstein if and only if \par\noindent
 $(d-2)/n \in \Z$.
In this case, $\br(\m^n) = 1+(d-2)/n$. 
\end{ex}

\par 
Now we will prove our Main Theorem \ref{Main: d-4,5}.

\begin{proof}[Proof of Theorem \ref{Main: d-4,5}] 
As in Proposition \ref{p:repMin}, we write 
\[Z = \phi^*(u C) + F = \phi^*( u C) + \sum_{i=1}^m a_i F_i.\]
Then, we have 
\begin{equation}
\label{eq:chi}
\chi(Z) = \chi(uC) + \chi( F ) = du (u - d + 2 ) / 2 + \sum_{i=1}^m a_i(a_i+1)/2.\end{equation}
By Corollary \ref{VerNTC} it suffices to classify the cases when $\br(I) =2$
 ($\chi(Z) = 0$) and then ask if there is some $W$ with $Z = sW$ for some $s \in \bbZ$. 
\par 
Since we must have $\chi(Z) \le 0$, we have $u \le d-2$ and 
if $u = d-2$, then $Z = \phi^*((n-2)C)$ (namely $I = \m^{d-2}$) by \ref{p:repMin}(1). 
Hence we have $I = \m$ if $d=3$ and we may assume $u \le d-3$ for $d = 4,5$. 
\par
Then, since $\br(I)=2$, we have $q( \infty I) = q(I) \ge q(u \m)  >0$ by \lemref{deg d hyp}  (5), 
we must have 
$Z E_0=0$ by  Lemma \ref{l:q(inf)}. 
\par 
Now we will go on for remaining $d, u$.

\begin{enumerate}
\item If $d=4, u=1$, then we have $\chi(C) = -2$. To have $\chi(Z)=0$, 
we must have $m=2, a_1=a_2=1$. But in this case, we have $Z E_0 \le 
- 4 + 2 < 0$  by \lemref{ZE_0}. Hence there is no $I=I_Z$ with $\chi(Z) =0$ and $\br(I)=2$.    
\par 
This shows if $d=4$, $\m$ and $\m^2$ are only examples that $\overline{G}(I)$ is Gorenstein.  

\item If $d =5, u = 2$, then we have $\chi(2C) = -5$. Since $(uC)C = -10$ in this case, 
we easily see that $Z E_0 < 0$   as above. Hence there is  no $I=I_Z$ with $\chi(Z) =0$ and $\br(I)=2$.   
\item If $d=5, u=1$, then we have $\chi(C) = -5$. We must have $\sum_{i=1}^m a_i(a_i-1)/2 = 5$
and since $Z  E_0=0$, we must have $\sum_{i=1}^m a_i\ge 5$  by \lemref{ZE_0}. This forces 
$m=5$ and $a_1=\ldots =a_5=1$. We show $I$ must be $I(L)$ defined in Definition \ref{I(L)def}. 
We can also show that actually, $\br(I)=2$.  
\end{enumerate}
This finishes our proof of Theorem \ref{Main: d-4,5}.
\end{proof}

We will see explicitly the ideal $I(L)$ appeared in Theorem \ref{Main: d-4,5}.

\begin{defn}\label{I(L)def}
Let $L\in A$ be a linear form, i.e., the image of a homogeneous polynomial of degree one in $k[X,Y,Z]$. 
We define an ideal $I(L)$ to be $(L)+\m^2$. 
It follows from Example 3.3 of \cite{Wt3} that $I(L)$ is integrally closed.
\end{defn}

\begin{lem}
\label{l:Lm}
Assume $Z_0=C$  (i.e., $u=1$), $ZE_0=0$ and that $m=d$ and $F=\sum_{i=1}^dF_i$ (i.e., $a_i=1$ for all $i$) in \proref{p:repMin}.

Then we have the following.

\begin{enumerate}
\item There exists a linear form $L\in I$  such that $I = I(L)$. 
\item $\chi(Z )=\frac{1}{2} d (5-d)$ and 
 $q(I)=-2 + d+ S(d-1)$.
In particular, if $d=5$, then $I$ is elliptic and $\ol G(I)$ is Gorenstein. 
\end{enumerate}
\end{lem}

\begin{proof}
(1)
Let $h\in I$ be a general element so that 
\[\di_X (h) = Z + H,\]
where $H$ does not contain exceptional components on $X$.
Note that $Z =  \phi^*(2C) - E_0$. Hence $h \not\in \m^2$ and $\m^2\subsetneq I=I_Z$. 
Take a linear form $L$ so that $L - h \in \m^2$.
Then 
$I(L) = \m^2 + (L) \subset I_Z$. 
Now, since $\ell_A(A/\m^2)=4$ and $\m^2 \subsetneq I_Z$, 
we have $\ell_A(A/I_Z) \le \ell_A(A/I(L)) = 3$. 
Recall that if $\psi : X_1\to X_0$ is a blowing-up of any point $p\in C$ and we put 
$Z_1 = \psi^*(C) + E$, where $E$ is the exceptional curve of $\psi$, then 
$Z_1$ is generated and $\ell_A(A/I_{Z_1})=2$. Since $I_Z$ is strictly contained in an 
ideal of the form $I_{Z_1}$, we must have $\ell_A(A/I_Z) = 3$ and $I_Z = \m^2 + (L)= I(L)$.   
\par
(2) We have $\chi(Z) =  d(5 -d)/2$ by \eqref{eq:chi} and  $q(I)=-2 + d+ S(d-1)$ by Lemma \ref{deg d hyp} (4), (5).
Now, assume that $d=5$. Then $\chi(Z) =0$ and $q( I(L)) =7$. 
We prove that $\br(I(L))=2$.
To show that, it suffices to show that 
$q( 2 I(L)) = q(I(L))$. \par
Since $p_g(A)=10$, $Z^2= -10$, and $\chi(2 Z) = 2 \chi(Z) - Z^2 = 10$, 
 we have 
\[\ell_A( A/ \overline{I(L)^2}) =  \chi(2Z) + [p_g(A) -q(2I)] \ge 10 + [p_g(A) - q(I(L))]=13.\]  
On the other hand, we easily compute that $\ell_A( A/ (I(L)^2) = 13$. 
Since $\ell_A(A/ \overline{(I(L)^2}) \le \ell_A( A/ (I(L)^2)$, we get 
$I(L)^2$ is integrally closed and that $q(2 I(L)) = q(I(L))$ or, $\br(I(L))=2$.   Hence $\ol G(I(L))$ is Cohen-Macaulay  by \cite{M} and \cite{Hun}.
Since $\chi(Z)=0$,  $\ol G(I(L))$ is Gorenstein by \thmref{Main}.
\end{proof}

\begin{rem} 
For any linear form $L$, the ideal $I(L)$ is represented by a cycle $Z$ satisfying the assumption of \lemref{l:Lm}, and therefore, $I(L)$ is elliptic and $\ol G(I(L))$ is Gorenstein if $d=5$.
To see this, we take $h:=L+h'\in \m$ with general element $h'\in \m^2$ so that $V_0:=\di_{X_0}(h)-C$, the proper transform of $\di_{\spec A}(h)$, is nonsingular.
Clearly, $V_0 C=d$ and $V_0 \cap C$ depends only on $L$. Taking $d$ blow-ups at the intersection of the proper transforms of $V_0$ and $C$, we have the cycle $Z=\phi^*C+\sum _{i=1}^d F_i$ such that $Z E_0=0$ and $\cO_X(-Z)$ is generated by $h$ and a general element of $\m^2$ (see \remref{r:ai} for the description of $F$).
\par 
When $C\cap V_0$ has distinct $d$ points, then $I(L)$ is a special case of $I_{Z_r}$ in \cite[4.5]{OWY5}.
\end{rem}

\begin{ex}
Let us show some concrete examples of $I(L)$ with $d=5$.
Let $\Gamma$ denote the weighted dual graph of $E$.

(1) Let $f=X^5+Y^5+Z^5$, $L=Y+Z$, $h=L+Y^2$ and $p=(0:1:-1)\in \PP^2$.
We have $\{p\}=V(f)\cap V(L)\subset \PP^2$.
Then $\cO_C(V(L))\cong \cO_C(5p)$ and  $\Gamma$ is as follows.

\begin{center}
\begin{picture}(20,30)(100,15)
\multiput(15,25)(40,0){6}{\circle*{4}}
\put(15,25){\line(1,0){200}}
\put(15,15){\makebox(0,0){$[6]$}}
\put(15,35){\makebox(0,0){$-10$}}
\put(55,35){\makebox(0,0){$-1$}}
\multiput(95,35)(40,0){4}{\makebox(0,0){$-2$}}
\put(55,15){\makebox(0,0){$E_5$}}
\put(95,15){\makebox(0,0){$E_4$}}
\put(135,15){\makebox(0,0){$E_3$}}
\put(175,15){\makebox(0,0){$E_2$}}
\put(215,15){\makebox(0,0){$E_1$}}
\end{picture}
\end{center}
We have 
\begin{gather*}
F_j=\sum_{i=j}^5 E_i, \quad 
Z= \phi^*(2C)-E_0 = E_5^*.
\end{gather*}

(2) Let $f=X^3(X-Y)^2+Y^5+Z^5$, $h=Y+Z+Y^2$ and $p=(0:1:-1), q=(1:1:-1)\in \PP^2$.
Then $\cO_C(V(L))\cong \cO_C(3p+2q)$ and  $\Gamma$ is as follows.

\begin{center}
\begin{picture}(20,30)(100,15)
\multiput(15,25)(40,0){6}{\circle*{4}}
\put(15,25){\line(1,0){200}}
\put(15,15){\makebox(0,0){$E_4$}}
\put(15,35){\makebox(0,0){$-2$}}
\put(55,35){\makebox(0,0){$-1$}}
\put(95,35){\makebox(0,0){$-10$}}
\put(135,35){\makebox(0,0){$-1$}}
\multiput(175,35)(40,0){2}{\makebox(0,0){$-2$}}
\put(55,15){\makebox(0,0){$E_5$}}
\put(95,15){\makebox(0,0){$[6]$}}
\put(135,15){\makebox(0,0){$E_3$}}
\put(175,15){\makebox(0,0){$E_2$}}
\put(215,15){\makebox(0,0){$E_1$}}
\end{picture}
\end{center}
We have 
\begin{gather*}
F_j=\sum_{i=j}^3 E_i \ \ (1 \le j \le 3), \quad 
F_4=E_4+E_5, \quad F_5=E_5, \\
Z= \phi^*(2C)-E_0 = E_3^* + E_5^*.
\end{gather*}
\end{ex}

\begin{acknowledgement}
The first-named author was partially supported by JSPS Grant-in-Aid 
for Scientific Research (C) Grant Number 21K03215. 
The second-named author was partially supported by JSPS Grant-in-Aid 
for Scientific Research (C) Grant Number 23K03040. 
The third named author was partially supported by JSPS Grant-in-Aid 
for Scientific Research (C) Grant Number 19K03430.
\par 
Furthermore, 
the authors thank Professor Jugal Verma for informing us that 
the equivalence of (1),(2) and (3) of Theorem \ref{Main} has already been proved by Heinzer et. al 
(\cite[Theorem 4.3]{HKU}). 
\end{acknowledgement}


\end{document}